\documentclass[10pt,a4paper,onecolumn,oneside]{article}
\usepackage[utf8]{inputenc}
\usepackage[english]{babel}

\usepackage[top=2cm,bottom=2cm,left=3cm,right=3cm]{geometry}
\usepackage{setspace}
\usepackage{enumitem}

\usepackage{amsmath,amsfonts,amssymb,amsthm,mathrsfs}
\allowdisplaybreaks[1]

\usepackage{hyperref}
\hypersetup{colorlinks=true,linkcolor=red!70!green!70!black,citecolor=red!70!green!70!black,urlcolor=magenta}
\usepackage{graphicx,color,xcolor,array}

\usepackage{tikz,pgfplots}
\pgfplotsset{compat=1.7}

\newtheorem{theorem}{Theorem}[section]
\newtheorem{lemma}[theorem]{Lemma}
\newtheorem{proposition}[theorem]{Proposition}
\newtheorem{corollary}[theorem]{Corollary}
\numberwithin{equation}{section}

\def\CC{\mathscr{C}}

\def\NN{\mathbb{N}}
\def\RR{\mathbb{R}}

\def\ZZ{\mathbb{Z}}

\newcommand{\nm}{\overline{N}}
\def\H{(H^1(\Omega))^3}
\newcommand{\pmm}{\overline{p}}
\newcommand{\pv}{p^{\circ}} % \p^{\perp}

\newcommand{\ulo}{U} % u^{LO}  % u^G
\newcommand{\um}{\overline{u}} % Qu
\newcommand{\uv}{u^{\circ}} % u^{\perp} % u^{\circ} % u^* % \dot{u} % u' % u^{\star} % u^{\dagger} 
\newcommand{\ver}{^{\circ}} % ^{\perp}  % ^{\circ} % ^{*}
\newcommand{\vm}{\overline{v}} % Qv
\newcommand{\vv}{v^{\circ}} % v^{\perp}

\newcommand{\wm}{\overline{\omega}} % Q\omega

\newcommand{\wlo}{W} % \omega^{LO}  % \omega^G
\newcommand{\wv}{\omega^{\circ}} % \omega^{\perp}
\newcommand{\wwm}{\overline{w}}

\DeclareMathOperator{\BS}{BS} % \newcommand{\BS}{\mathrm{BS}}
\DeclareMathOperator{\di}{div}
\DeclareMathOperator{\rot}{rot}

\begin{document}

\author{Quentin \textsc{Vila}\footnote{
Univ. Grenoble Alpes, CNRS, Institut Fourier, 38000 Grenoble, France
}}
\title{ \LARGE \bf
Stability of Three-dimensional Oseen Vortices under Helical Perturbations
}
\date{{\normalsize 2024, January, the 16$^{th}$}}

\maketitle

\begin{center}
{\bf Abstract.}
\end{center}

We study the long-time behaviour of helically symmetric infinite-energy solutions to the incompressible Navier-Stokes equations in the whole space $\RR^3$. Our solutions are $H^1$-perturbations of a Lamb-Oseen vortex whose circulation Reynolds number can be any fixed real number. If $v$ denotes the helical velocity perturbation, no matter how large at initial time in $H^1(\RR^3)$, we show that the scale-invariant quantities $\|v(t)\|_{L^2}$ and $\sqrt{t} \|\nabla v(t)\|_{L^2}$ converge to zero as $t \to +\infty$.
This proves that the Oseen vortex is globally stable with respect to $H^1$-helical perturbations. Our analysis relies on a logarithmic energy estimate for the perturbation $v$, on the Ladyzhenskaya inequality for helical vector fields, and on Poincar\'e's inequality which implies an exponential decay in time for the velocity components whose mean value is zero along the symmetry axis.


\section{Introduction}

We consider the evolution of helical infinite-energy solutions to the incompressible Navier-Stokes equations in $\RR^3$
\begin{equation}\label{eqNSdimensionless}
\left\{\begin{array}{l}
\partial_tu + (u\cdot\nabla)u = \Delta u - \nabla p\\
\rule[1.3em]{0pt}{0pt}
\di u = 0
\end{array}\right.
\end{equation}
where the velocity $u:\RR\times\RR^3\to\RR^3$ and the pressure $p:\RR\times\RR^3\to\RR$ are considered dimensionless, and the kinematic viscosity has been set to one without loss of generality.

Let $L>0$ denote a fixed positive real number. 
We will say that a scalar map $f$ is \emph{helical} if $f$ is constant along vertical helices of pitch $2\pi L$. 
In the three-dimensional cylindrical coordinate system $(r,\theta,z)$, $f$ is helical if and only if $f(r,\theta,z) = f(r,\theta-\frac{z}{L},0)$ for every $r\in\RR_+$, $\theta\in\RR/2\pi\ZZ$, $z\in\RR$. If $f$ is differentiable, this is equivalent to the relation $(\partial_\theta + L\partial_z)f = 0$. For a vector field $u$, being helical means that its cylindrical components
are all helical:
\begin{equation}\label{eqHelical}
u(t,r,\theta,z) = u_r(t,r,\varphi)e_r + u_\theta(t,r,\varphi)e_\theta + u_z(t,r,\varphi)e_z
\end{equation}
where
\[
\varphi = \theta - \tfrac{z}{L} \in\RR/2\pi\ZZ
\quad\text{ and }\quad
e_r = \begin{pmatrix}\cos\theta\\\sin\theta\\0\end{pmatrix},\ 
e_\theta = \begin{pmatrix}-\sin\theta\\\cos\theta\\0\end{pmatrix},\ 
e_z = \begin{pmatrix}0\\0\\1\end{pmatrix}.
\]
In particular, a helical vector field is $2\pi L$-periodic along the vertical axis. We will hence study the Navier-Stokes equations \eqref{eqNSdimensionless} in the vertically-periodic domain 
\[
\Omega = \{(r,\theta,z),\ r\in\RR_+,\, \theta\in\RR/2\pi\ZZ,\, z\in\RR/2\pi L\ZZ\}
\]
which we will think of as $\RR^2\times\RR/2\pi L\ZZ$ in cylindrical coordinates.

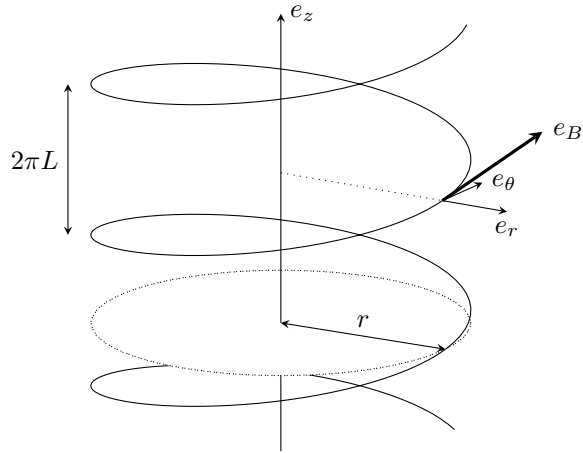
\begin{figure}[hbt]
\centering
\begin{tikzpicture}

\draw [variable=\t,domain=-1.7:0,samples=70] plot({2.5*cos((\t-1/6)*180)},{0.7*sin((\t-1/6)*180)+\t}) ;

\draw [variable=\t,domain=0:2,samples=70] plot({2.5*cos(\t*360)},{0.7*sin(\t*360)}) [dotted,fill=white] ;
\draw [>=stealth,<->] (0,0) -- node[midway][above]{$r$} (2.5*1.73205/2,-0.7*1/2) ; 
\draw (0,-1.7) -- (0,-0.7) ;
\draw [>=stealth,->] (0,0) -- (0,4.1) node[right]{$e_z$} ;

\draw [variable=\t,domain=0:4.1,samples=140] plot({2.5*cos((\t-1/6)*180)},{0.7*sin((\t-1/6)*180) + \t}) ;
\draw [>=stealth,<->] ({-2.5-0.3,1+1/6}) -- ++(0,2) node[midway][left]{$2\pi L$} ;

\draw [dotted] (0,1.99) -- ++({2.5*cos((1.99-1/6)*180)},{0.7*sin((1.99-1/6)*180)}) ;
\draw [>=stealth,->] ({2.5*cos((1.99-1/6)*180)},{0.7*sin((1.99-1/6)*180) + 1.99}) -- ++({1*cos((1.99-1/6)*180)},{0.7/2.5*sin((1.99-1/6)*180)}) node[below]{$e_r$}  ;
\draw [>=stealth,->] ({2.5*cos((1.99-1/6)*180)},{0.7*sin((1.99-1/6)*180) + 1.99}) -- ++({-1*sin((1.99-1/6)*180)},{0.7/2.5*cos((1.99-1/6)*180)}) node[right]{$e_\theta$}  ;
\draw [>=stealth,->,very thick] ({2.5*cos((1.99-1/6)*180)},{0.7*sin((1.99-1/6)*180) + 1.99}) -- ++({-2.5*sin((1.99-1/6)*180)},{0.7*cos((1.99-1/6)*180) + 1/pi}) node[right]{$e_B$}  ;

\end{tikzpicture}
\caption{
The tangent vector to a helical curve of radius $r$ and pitch $2\pi L$ is the Beltrami vector $e_B = re_\theta + Le_z$, which satisfies $e_B\cdot\nabla = \partial_\theta + L\partial_z$ in cylindrical coordinates. A differential scalar map $f$ is helical if $e_B\cdot\nabla f = 0$.
}
\end{figure}

The helical symmetry in the study of three-dimensional flows recently raised a growing interest, following the 1990's work of Titi, Mahalov and Leibovich \cite{MahaTitiLeib} where the authors  show the global well-posedness of the helical Navier-Stokes equations in $H^1$ in an infinite vertical cylinder with homogeneous Dirichlet boundary conditions.
In 1999, Dutrifoy shows in \cite{DutriEulerHeli} that the helical Euler equations are also globally well-defined for regular initial data in any horizontally-bounded helical domain, provided that the flow $u$ has no helical swirl, that is to say, satisfies $u \cdot (re_\theta + Le_z) = ru_\theta + Lu_z = 0$.
Ten years later, the same result is proved by Ettinger and Titi for some class of weak solutions, see \cite{EttingerTiti}.

In the last decade, one can note the work of Lopes Filho, Nussenzveig Lopes and their collaborators on both the Euler and the Navier-Stokes equations, see for example  \cite{LoMaNiuLoTiti,JiuLoNiuLo}.
In the first article \cite{LoMaNiuLoTiti}, the authors consider a fluid either viscous or inviscid evolving in an infinite vertical circular pipe and observe that the fluid behaves as a planar flow when the helical pitch tends to infinity. 
In the second article \cite{JiuLoNiuLo}, the authors show that the helical Navier-Stokes equations are well-posed in the whole three-dimensional space for any helical initial data in $H^1$, and remark under some hypothesis on the helical swirl (relative to the value of the viscosity) that, when the viscosity tends to zero, the viscous helical flow converges towards the inviscid helical flow with zero helical swirl.
The fact that the helical swirl of the flow does not need to be zero when considering the Navier-Stokes equations can be proven thanks to the helical Ladyzhenskaya inequality, see \cite{MahaTitiLeib} and \cite{Ladyzhen} (or Lemma \ref{lemmaLadyzhen} below), which is inherited from the two-dimensional analysis. This inequality ensures the global well-posedness of the helical Navier-Stokes problem for initial data of any size. 

The goal of the present article is to study the asymptotic behaviour of helical solutions to \eqref{eqNSdimensionless} when the time tends to infinity.
It is already known that viscous fluids with finite energy return at large times to the hydrostatic equilibrium, see for example \cite{Masuda84}.
We aim here to establish that this result remains true for any helical flow of finite energy interacting with one Oseen vortex. 
Despite the Oseen vortex carrying an infinite amount of kinetic energy, the way it interacts with some helical perturbation in $H^1$ does not prevent it from remaining asymptotically stable nor the $H^1$ perturbation from returning to equilibrium.

The \emph{Lamb-Oseen vortex} is a self-similar flow with Gaussian profile which arises naturally when studying the dynamics of the two-dimensional Navier-Stokes equations, as it can be seen for example in \cite{MajBert,GalWayR2,GalWay05,GalMae12}.
In  three dimensions, when studying the Navier-Stokes-Coriolis system, the Lamb-Oseen vortex appears as the limit of some vertically-periodic flows of infinite energy rotating at high speed, as explained in \cite{GallayRoussierMichon}.
In the three-dimensional case, the Lamb-Oseen vortex $\wlo$ and its associated velocity field $\ulo$ can be defined as 
\begin{align}
\label{equlo}
&\ulo(t,r,\theta,z) = \frac{1}{2\pi r}\left(1 - e^{-\frac{r^2}{4(1+t)}}\right)e_\theta, 
\\\rule[2.2em]{0pt}{0pt}
\label{eqwlo}
&\wlo(t,r,\theta,z) = \frac{1}{4\pi(1+t)}e^{-\frac{r^2}{4(1+t)}}\ e_z.
\end{align}
These three-dimensional vector fields are \emph{radial}, \emph{ie.} their cylindrical components only depend on the radial variable $r$. 
In particular they are invariant by translation along the vertical axis, which is what Majda and Bertozzi call \emph{two-and-a-half dimensional flows} \cite{MajBert}, and they are helical. 
Moreover, the velocity field $\ulo$ associated with the Lamb-Oseen vortex is a self-similar solution to the Navier-Stokes equations \eqref{eqNSdimensionless}, and it has infinite kinetic energy in the sense that its $L^2$-norm is infinite.

Let us fix a real number $a\in\RR$ and introduce the functional spaces 
\begin{equation}
H = \{v\in (H^1(\Omega))^3 \mid \di(v)=0,\, v\text{ is helical} \}
\end{equation}
and
\begin{equation}\label{defX}
X = H + a\ulo(0) 
= \{v+a\ulo(0),\ v\in H\},
\end{equation}
where
\[
\|v\|_{H^1}^2 = \|v\|_{L^2}^2 + \|\nabla v\|_{L^2}^2
\ ,\qquad
\|v\|_{L^2} = \left(\int_\RR\int_0^{2\pi}\!\!\int_0^{2\pi L} |v(r,\theta,z)|^2\,rdrd\theta dz\right)^{\frac{1}{2}},
\]
and where $\ulo(0,r,\theta,z) = \frac{1}{2\pi r}(1 - e^{-\frac{r^2}{4}})e_\theta$.
The real number $a$ will be refered to as the \emph{circulation Reynolds number} of the Oseen vortex.
The goal of this article is to show that for any initial data $u_0$ in $X$, the Navier-Stokes flow is globally well-defined in $X$ and tends towards the Oseen vortex as time goes to infinity. More precisely, we prove the following theorem. 

\begin{theorem}\label{ThIntro}
Let $a\in\RR$. Let $u_0\in X$. Equation \eqref{eqNSdimensionless} with initial data $u_0$ has a unique global solution $u(t) = v(t) + a\ulo(t)$ where $v\in\CC^0([0,+\infty[\,,H)$. 
Moreover, the following limit holds
\begin{equation}\label{eqThIntro}
\|u(t) - a\ulo(t)\|_{L^2} + \sqrt{t}\,\|\nabla u(t) - a\nabla\ulo(t)\|_{L^2} \xrightarrow[t\to+\infty]{} 0.
\end{equation}
\end{theorem}

\paragraph{Remarks.}
\begin{enumerate}
[leftmargin=0pt,itemindent=!]
\item Given the properties of $\ulo$ listed in Lemma \ref{lemmaUloIneq} below, we have naturally that $X = H + a\ulo(0) = H + a\ulo(t)$ for any $t\geqslant0$. Theorem \ref{ThIntro} asserts consequently the existence of a solution remaining in $X$ for all times. We chose to define the functional space $X$ in \eqref{defX} from $\ulo(0)$ to highlight the fact that $X$ does not depend on time, but $\|u(t) - a\ulo(0)\|_{L^2}$ and $\sqrt{t}\,\|\nabla u(t) - a\nabla\ulo(0)\|_{L^2}$ do not tend to zero in general. This the reason why we shall always consider the difference $v(t) = u(t) - a\ulo(t)$ in what follows, as in Theorem \ref{ThIntro}, rather than $u(t) - a\ulo(0)$.
\item When $a=0$, Theorem \ref{ThIntro} treats the case where $u_0$ has finite energy and states that the helical Navier-Stokes equations are globally well-posed in $H^1$ (as shown in \cite{JiuLoNiuLo}) and that the kinetic energy of the fluid tends to zero, without any smallness assumption on the initial data (as shown in \cite{Masuda84} for all weak solutions).
\item When $a\neq0$, Theorem \ref{ThIntro} says that the Oseen vortex $a\ulo$ remains asymptotically stable under any helical perturbation in $H^1(\Omega)$, in the following sense. The Oseen vortex $a\ulo$ is the solution of \eqref{eqNSdimensionless} with initial data $a\ulo(0)$, and Theorem \ref{ThIntro} says that for any $v_0\in H$, the solution $u$ of \eqref{eqNSdimensionless} with initial data $v_0+a\ulo(0)$ is equivalent to $a\ulo$ at large times: $u(t)\sim a\ulo(t)$ when $t\to+\infty$.
\item The circulation Reynolds number $a$ of the Oseen vortex is a significant parameter of the problem, see for example \cite{GalWay05,GalMae12}, which is invariant under rescaling.
As discussed in \cite{GalWay05}, many situations in hydrodymanics see instabilities appearing at high Reynolds numbers, and in \cite{GalMae12} for example, the stability of the two-dimensional Oseen vortex is shown in exterior domains only when the circulation Reynolds number is small enough. In contrast, Theorem \ref{ThIntro} holds without any smallness assumption on $a\in\RR$.
\item Theorem \ref{ThIntro} applied to two-and-a-half dimensional flows implies as a corollary that the limit \eqref{eqThIntro} holds for any two-dimensional radial perturbation $v_0$ in $H^1$, while it is not known if this is true in the general two-dimensional case when the flow is not supposed radial.
\item The proof of Theorem \ref{ThIntro} actually gives more detailed information on the long-time behaviour of the solutions of \eqref{eqNSdimensionless} in $X$. 
If we decompose some helical three-dimensional viscous flow $\mbox{$u:t\mapsto u(t)\in X$}$ into a zero-vertical-mean part $\uv$ plus its orthogonal component $\um$ as below, we remark that the time-asymptotic expansion of $u$ is the same as that of a radial two-and-a-half-dimensional viscous flow.
Indeed, let us define the orthogonal projection operator $Q$ as
\begin{equation}\label{eqdefQ}
\forall v\in (H^1(\Omega))^3,\qquad Qv = \frac{1}{2\pi L}\int_0^{2\pi L} v(\cdot,\cdot,z)\,dz\,;
\end{equation}
meaning that $1-Q$ is the orthogonal projection operator on the vector fields of $(H^1(\Omega))^3$ whose mean value relative to the vertical coordinate is zero.
Writing for conciseness 
\begin{equation}\label{eqdefumuv}
\um = Qu \qquad\text{ and }\qquad \uv = (1-Q)u,
\end{equation}
we thus have $u = \um + \uv$ where $\um = \vm + a\ulo \in X$ and $\uv = \vv\in H$. 
Because $u$ is helical, $\um$ and $\vm$ are radial two-and-a-half dimensional vector fields. 
On the other hand, because $\uv$ has a zero vertical mean value, we will show in Proposition \ref{lemmaEstimatesuv} below that $t\mapsto\|\uv(t)\|_{H^1}$ decays exponentially.
As a consequence, the long-time behaviour of $u$ is entirely governed by the equation of $\vm$, which is the heat equation 
\[
\partial_t\vm + \nm = \Delta\vm
\]
where the source term $\nm$ decays exponentially when time tends to infinity. Knowing this, one could determine further terms in the long-time asymptotic expansion of $u$ thanks to what is already known about the heat equation.
This point will be developped in Section \ref{SectionAsymptotics} and discussed in Section \ref{SectionConclusion}. 
\end{enumerate}

The rest of the paper is organised as follows. 
In Section \ref{SectionExistence} we prove the global well-posedness of the Navier-Stokes equations \eqref{eqNSdimensionless} in $X$, using the Ladyzhenskaya inequality for helical vector fields.
In Section \ref{SectionEstimates}, we decompose the velocity field $u = \um + \uv$ into its vertical mean value $\um$ and its zero-vertical-mean part $\uv$ and we give different types of estimates on these components. In particular, we see that the $L^2$-norms of the zero-vertical-mean part $\uv$ and its derivatives decay exponentially in time. The key points of this analysis are the Poincaré inequality in $\Omega$ and an improved energy estimate due to Gallay and Maekawa \cite{GalMae12}.
In Section \ref{SectionAsymptotics} we prove the convergence result \eqref{eqThIntro} in Theorem \ref{ThIntro}.
Finally, in Section \ref{SectionConclusion} we suggest to study the long-time behaviour of $u$ from the point of view of the vorticity $\omega = \rot(u)$ in some weighted $L^2$ spaces in order to obtain decay rates or further asymptotic terms in the long-time asymptotic expansion of $u$.

\section{Global Well-Posedness of the Helical Navier-Stokes Equations}
\label{SectionExistence}

This section follows the analysis led by Mahalov, Titi and Leibovich in \cite{MahaTitiLeib}. 
In particular, it will only slightly differ from the proofs that can be found there, or later in \cite{JiuLoNiuLo} for the whole space. 

To begin with, we need to state some properties of the velocity field associated with the Oseen vortex. All of them stem from direct calculations, see for example \cite{GalMae12}.

\begin{lemma}[$L^2$ and $L^\infty$ estimates for the Oseen vortex]\label{lemmaUloIneq}
There exists a constant $C>0$ such that for every $t \geqslant 0$
\[
\|\ulo(t)\|_{L^\infty} \leqslant \frac{C}{\sqrt{1+t}}
\quad\text{ and }\quad
\|\nabla\ulo(t)\|_{L^\infty} \leqslant \frac{C}{1+t},
\]
and for every $t_2 \geqslant t_1 \geqslant 0$
\begin{align}\label{eqMaekawaIneq1}
\|\ulo(t_2) - \ulo(t_1)\|_{L^2}^2 
&\leqslant CL\,\ln\!\left(\frac{1+t_2}{1+t_1}\right)
\\\label{eqMaekawaIneq2}
\text{and }\qquad
\|\nabla\ulo(t_2) - \nabla\ulo(t_1)\|_{L^2}^2 
&\leqslant CL \left(\frac{1}{1+t_1} - \frac{1}{1+t_2}\right).
\end{align}
\end{lemma}

As a consequence of these $L^2$ estimates, for every $t\geqslant 0$ one has $\ulo(t) - \ulo(0) \in H$ and $a\ulo(t)\in X$ where $X$ and $H$ are defined in \eqref{defX}. 
For this reason, if $u:t\mapsto u(t)\in X$ is a vector field in $X$ depending on time, then for any $t\geqslant0$ we have the decomposition $u(t) = v(t) + a\ulo(t)$ with $v(t)\in H$, by writing $v(t) = (u(t) - a\ulo(0)) + a(\ulo(0) - \ulo(t)) \in H$.
Given that $\ulo$ is a solution of the heat equation $\partial_t\ulo = \Delta\ulo$, we have that $u$ satisfies \eqref{eqNSdimensionless} in $X$ if and only if $v$ satisfies in $H$ the equation
\begin{equation}\label{eqv}
\left\{\begin{array}{l}
\partial_tv + v\cdot\nabla v + a(\ulo\cdot\nabla v + v\cdot\nabla \ulo) = \Delta v - \nabla p - a^2\ulo\cdot\nabla\ulo\\
\rule[1.3em]{0pt}{0pt}
\di v = 0
\end{array}\right.
\end{equation}
where
\[
\ulo\cdot\nabla\ulo = -\nabla p^{\ulo} \quad\text{ for }\quad p^{\ulo}(t,r) = \int_0^r \frac{1}{\rho}\ulo_\theta(t,\rho)^2\,d\rho.
\]

Because on the one hand $\ulo(0)$ is in $(L^p(\Omega))^3$ for $p>2$ and on the other hand the Sobolev embeddings give $(H^1(\Omega))^3 \subset (L^p(\Omega))^3$ for $2\leqslant p\leqslant 6$, the classical results on the Navier-Stokes equations already tell us that from any initial data $u_0\in X$ there exists a unique local mild solution of \eqref{eqNSdimensionless} for example in $(L^3(\Omega))^3$ (see \cite{Kato84}), which is the Lebesgue critical space. 
We need however to verify that this solution remains in $X$. 
The fact that $u(t)$ remains helical for all times is due to the symmetries of the Navier-Stokes equations, see for example \cite{JiuLoNiuLo}, because equations \eqref{eqNSdimensionless} describe the evolution of a fluid in a homogeneous and isotropic physical medium. To prove that $v(t) = u(t) - a\ulo(t)$ remains in $(H^1(\Omega))^3$, we will refer to the integral version of equation \eqref{eqv} and follow the usual two-steps argument: using the Banach fixed-point theorem to show that $v$ remains in $H^1$ at least locally around $t=0$, then exhibiting a bound on $\|v(t)\|_{H^1}$ that does not blow up in finite time.

Denoting by $P$ the Leray projection operator on the divergence-free vector fields, the integral equation associated to \eqref{eqv} is 
\begin{equation}\label{eqvIntegral}
v(t) = S(t)v_0 - \int_0^tS(t-s)P\left[\,v(s)\cdot\nabla v(s) + a\ulo(s)\cdot\nabla v(s) + av(s)\cdot\nabla\ulo(s)\,\right]\,ds
\end{equation}
where $(S(t))_{t\geqslant0}$ denotes the heat semigroup in $\Omega$.
We thus have the following local existence result.

\begin{proposition}[Local existence]\label{ThLocalExistence}
Let $v_0\in \H$. There is a time $T>0$ such that there exists a unique solution $v\in \CC^0([0,T],\H) \cap L^2([0,T],(H^2(\Omega))^3)$ to equation \eqref{eqvIntegral}. 
Moreover, if the maximal time $T'$ of existence of such a solution is finite, then $\|v(t)\|_{H^1} \longrightarrow +\infty$ as $t\to T'$.
\end{proposition}

\begin{proof}
The proof follows the classical approach of Fujita and Kato \cite{FujitaKato}. 
\end{proof}

We now have to prove that the solution $v$ of equation \eqref{eqvIntegral} is globally well-defined by showing that $\|v(t)\|_{H^1}$ does not blow up in finite time. For this, we will use classical energy estimates. As discussed in \cite{MahaTitiLeib}, the known difficulty of this step for the general three-dimensional Navier-Stokes equations disappears when considering helical flows, for in this case the analysis is similar to that of the two-dimensional Navier-Stokes equations.
Indeed, the inequality of Ladyzhenskaya which is used in asserting the global existence of two-dimensional flows $\|f\|_{L^4} \leqslant C\|f\|_{L^2}^{\frac{1}{2}}\|\nabla f\|_{L^2}^{\frac{1}{2}}$ has different exponents when the flow is three-dimensional $\|f\|_{L^4} \leqslant C\|f\|_{L^2}^{\frac{1}{4}}\|\nabla f\|_{L^2}^{\frac{3}{4}}$ (cf the first pages of her book \cite{Ladyzhen}), which prevents from concluding in the same way; however, when the flow is helical we can use an inequality with the same exponents as in the two dimensional case. 
The proof can be found in \cite{MahaTitiLeib} or \cite{JiuLoNiuLo}.

\begin{lemma}[Ladyzhenskaya's inequality for helical maps]\label{lemmaLadyzhen}
There exists a universal constant $C_0>0$ such that $\|v\|_{L^4} \leqslant (\frac{C_0}{L})^{\frac{1}{4}}\, \|v\|_{L^2}^{\frac{1}{2}} \|\nabla v\|_{L^2}^{\frac{1}{2}}$ for any helical vector field $v\in\H$.
\end{lemma}

We can now conclude the global well-posedness of \eqref{eqv} in $H$, and so of \eqref{eqNSdimensionless} in $X$.

\begin{theorem}[Global existence]\label{ThGlobal}
Let $v_0\in H$. Equation \eqref{eqvIntegral} has a unique global solution $v$ in $\CC^0([0,+\infty[\,,H)$.
\end{theorem}

\begin{proof}
From the properties of the heat semigroup $(S(t))_{t\geqslant0}$ we know that $v$ is regular once $t>0$, and therefore satisfies equations \eqref{eqv} as long as \eqref{eqvIntegral}  holds. 
We can hence derive from \eqref{eqv} the following energy estimates.

Let us first consider the scalar product in $(L^2(\Omega))^3$ of \eqref{eqv} against $v$, where
\[
\int_\Omega v(t)\cdot\partial_tv(t) = \int_\Omega \partial_t\left(\frac{1}{2}v(t)^2\right) = \frac{1}{2}\partial_t\|v(t)\|_{L^2}^2
\quad\text{ and }\quad
\int_\Omega v(t)\cdot\Delta v(t) = -\|\nabla v(t)\|_{L^2}^2
\]
by doing an integration by parts to get the right-hand-side equality; where
\[
\int_\Omega v(t)\cdot\nabla(p(t) - a^2p^\ulo(t)) = 0\,, \quad
\int_\Omega v(t)\cdot(v(t)\cdot\nabla)v(t) = 0 
\quad\text{ and }\quad
\int_\Omega v(t)\cdot(\ulo(t)\cdot\nabla)v(t) = 0
\]
by integrating by parts and because $\di v(t)=0$; and where
\[
\left|\int_\Omega v(t)\cdot(v(t)\cdot\nabla)\ulo(t)\right| 
\leqslant \|v(t)\|_{L^2}\|v(t)\cdot\nabla\ulo(t)\|_{L^2}
\leqslant \|v(t)\|_{L^2}^2 \|\nabla\ulo(t)\|_{L^\infty}
\leqslant \frac{C}{1+t}\|v(t)\|_{L^2}^2
\]
by using the Cauchy-Schwarz inequality, $C>0$ being a universal constant given by Lemma \ref{lemmaUloIneq}.
It thus holds for all $s>0$
\[
\frac{1}{2} \partial_s\|v(s)\|_{L^2}^2 \leqslant -\|\nabla v(s)\|_{L^2}^2 + \frac{1}{2}\frac{2C|a|}{1+s}\|v(s)\|_{L^2}^2\,,
\]
which we multiply by the factor $(1+s)^{-2C|a|}$ and integrate for $s$ between 0 and $t>0$ to get the classical energy estimate
\begin{equation}\label{eqEstimatev}
\|v(t)\|_{L^2}^2 + 2\int_0^t \|\nabla v(s)\|_{L^2}^2 \,ds \leqslant \|v_0\|_{L^2}^2 (1+t)^{2C|a|}.
\end{equation}
We did not use so far that $v$ is helical.

Secondly, we consider the scalar product of \eqref{eqv} against $\Delta v$. We write
\[
\int_\Omega \Delta v(t)\cdot\partial_tv(t) = -\frac{1}{2}\partial_t\|\nabla v(t)\|_{L^2}^2
\quad\text{ and }\quad
\int_\Omega \Delta v(t)\cdot\Delta v(t) = \|\Delta v(t)\|_{L^2}^2
\]
by using an integration by parts to get the left-hand side equality, and we see that $\int_\Omega \Delta v(t)\cdot\nabla(p(t) - a^2p^\ulo(t)) = 0$ because $\di(\Delta v(t)) = \Delta\di v(t) = 0$.
We then use the Cauchy-Schwarz and Young's inequalities to get
\begin{multline*}
\left|\int_\Omega \Delta v(t)\cdot a(v(t)\cdot\nabla)\ulo(t)\right| 
\leqslant \|\Delta v(t)\|_{L^2}|a|\|v(t)\cdot\nabla\ulo(t)\|_{L^2}
\\
\leqslant \|\Delta v(t)\|_{L^2} |a|\|v(t)\|_{L^2}\|\nabla\ulo(t)\|_{L^\infty}
\leqslant \frac{1}{4}\|\Delta v(t)\|_{L^2}^2 + \frac{C^2|a|^2}{(1+t)^2}\|v(t)\|_{L^2}^2
\end{multline*}
and
\begin{multline*}
\left|\int_\Omega \Delta v(t)\cdot a(\ulo(t)\cdot\nabla)v(t)\right| 
\leqslant \|\Delta v(t)\|_{L^2}|a|\|\ulo(t)\cdot\nabla v(t)\|_{L^2}
\\
\leqslant \|\Delta v(t)\|_{L^2} |a|\|\ulo(t)\|_{L^\infty}\|\nabla v(t)\|_{L^2}
\leqslant \frac{1}{4}\|\Delta v(t)\|_{L^2}^2 + \frac{C^2|a|^2}{(1+t)}\|\nabla v(t)\|_{L^2}^2
\end{multline*}
where $C>0$ is a universal constant given by Lemma \ref{lemmaUloIneq}.
Finally, because $\di v=0$ and using the Ladyzhenskaya inequality for helical maps (lemma \ref{lemmaLadyzhen}) we have that
\begin{multline*}
\left|\int_\Omega \Delta v(t)\cdot (v(t)\cdot\nabla)v(t)\right| 
\leqslant \int_\Omega |\nabla v(t)|^3
\leqslant \|\nabla v(t)\|_{L^4}^2\|\nabla v(t)\|_{L^2}
\\
\leqslant \left(\!\frac{C_0}{L}\!\right)^{\frac{1}{2}}\|\Delta v(t)\|_{L^2} \|\nabla v(t)\|_{L^2}^2
\leqslant \frac{1}{4}\|\Delta v(t)\|_{L^2}^2 + \frac{C_0}{L}\,\|\nabla v(t)\|_{L^2}^4.
\end{multline*}
Considering all of the above computation, for $s>0$ we have
\[
\frac{1}{2} \partial_s\|\nabla v(s)\|_{L^2}^2 \leqslant -\frac{1}{4}\|\Delta v(s)\|_{L^2}^2 + \frac{C^2|a|^2}{(1+s)^2}\|v(s)\|_{L^2}^2 + \frac{1}{2}K(s)\|\nabla v(s)\|_{L^2}^2
\]
where 
\begin{equation}\label{eqdefK}
K(s) = 2\frac{C^2|a|^2}{1+s} + 2\frac{C_0}{L}\|\nabla v(s)\|_{L^2}^2
\end{equation}
is locally integrable on $[0,+\infty[$ according to \eqref{eqEstimatev}, in such a way that $t\mapsto\int_0^tK(s)ds$ does not blow up in finite time. Multiplying this last differential inequality by $\exp(-\int_0^sK(\sigma)d\sigma)$ then integrating it in time for $s$ between 0 and $t>0$ gives that
\begin{equation}\label{eqEstimateGradv}
\|\nabla v(t)\|_{L^2}^2 + \frac{1}{2}\int_0^t \|\Delta v(s)\|_{L^2}^2 \,ds \leqslant \|\nabla v_0\|_{L^2}^2 \,e^{\int_0^t K(\sigma)d\sigma} + \int_0^t e^{\int_s^t K(\sigma)d\sigma}\frac{2C^2|a|^2}{(1+s)^2}\|v(s)\|_{L^2}^2\,ds
\end{equation}
where the member on the right-hand side does not blow up in finite time.

Thanks to \eqref{eqEstimatev} and \eqref{eqEstimateGradv} we can give a bound on $\|v(t)\|_{H^1}$ that does not blow up in finite time, which ensures that the solution is global. 
\end{proof}

The bound on $\|v(t)\|_{H^1}$ given in the previous proof by \eqref{eqEstimatev} and \eqref{eqEstimateGradv} tends towards infinity when $t\to+\infty$. The point of the following sections will be to write better energy estimates on $v$, and finally show that $\|v(t)\|_{H^1}$ tends towards 0 as $t\to+\infty$ in the sense of Theorem \ref{ThIntro}.

\section{Estimates for the Zero-Vertical-Mean Component}
\label{SectionEstimates}

Thanks to estimates \eqref{eqMaekawaIneq1} and \eqref{eqMaekawaIneq2} from Lemma \ref{lemmaUloIneq}, it is possible to refine inequality \eqref{eqEstimatev}. Indeed, due to Gallay and Maekawa \cite{GalMae12}, one has the following energy estimate.

\begin{proposition}[Logarithmic energy estimate for the $H^1$ component]\label{propMaekawaIneq}
The solution $v$ of \eqref{eqv} satisfies the following energy estimate
\begin{equation}\label{eqMaekawaIneq}
\|v(t)\|_{L^2}^2 + 2\int_0^t \|\nabla v(s)\|_{L^2}^2\,ds \leqslant C (1+\ln(1+t))
\end{equation}
for all $t\geqslant0$, where the constant $C = C(L,a,\|v_0\|_{L^2}) > 0$ is an increasing function of $L$, $a$, and the initial data $\|v_0\|_{L^2}$.
\end{proposition}

\begin{proof}
Let us consider the functions $\ulo_\tau(t) = \ulo(t+\tau)$ and $v'(t)= v(t) + a(\ulo(t) - \ulo_\tau(t))$ for some fixed $\tau\geqslant0$. We thus have $u = v' + a\ulo_\tau$, so $v'$ satisfies the equation
\[
\partial_tv' + v'\cdot\nabla v' + a(\ulo_\tau\cdot\nabla v' + v'\cdot\nabla \ulo_\tau) = \Delta v' - \nabla p'
\]
obtained in the same way as \eqref{eqv}, for some pressure $p'$ incorporating the term $a^2\ulo_\tau\cdot\nabla\ulo_\tau$. Following the calculation leading to inequality \eqref{eqEstimatev}, we get that for all $t\geqslant0$
\[
\|v'(t)\|_{L^2}^2 + 2\int_0^t \|\nabla v'(s)\|_{L^2}^2 \,ds \leqslant \|v'_0\|_{L^2}^2 \left(\frac{1+t+\tau}{1+\tau}\right)^{2C|a|}.
\]
On the other hand, estimates \eqref{eqMaekawaIneq1} and \eqref{eqMaekawaIneq2} give that for every $t\geqslant0$
\[
\|v(t)\|_{L^2}^2 
\leqslant 2\|v'(t)\|_{L^2}^2 + 2a^2\|\ulo(t)-\ulo_\tau(t)\|_{L^2}^2
\leqslant 2\|v'(t)\|_{L^2}^2 + 2CLa^2\ln\left(\frac{1+t+\tau}{1+t}\right)
\]
and
\begin{align*}
\int_0^t\|\nabla v(s)\|_{L^2}^2\,ds 
&\leqslant 
\int_0^t 2\|\nabla v'(s)\|_{L^2}^2 + 2a^2\|\nabla\ulo(s)-\nabla\ulo_\tau(s)\|_{L^2}^2\,ds
\\&\leqslant 
\int_0^t 2\|\nabla v'(s)\|_{L^2}^2 + 2CLa^2\left(\frac{1}{1+s} - \frac{1}{1+s+\tau}\right)\,ds
\\&\leqslant 
\int_0^t 2\|\nabla v'(s)\|_{L^2}^2\,ds + 2CLa^2\ln(1+t)
\end{align*}
and 
\[
\|v'_0\|_{L^2}^2 
\leqslant 2\|v_0\|_{L^2}^2 + 2a^2\|\ulo(0)-\ulo_\tau(0)\|_{L^2}^2
\leqslant 2\|v_0\|_{L^2}^2 + 2CLa^2\ln(1+\tau).
\]
By combining these four inequations and setting $t = \tau$ we deduce directly the desired estimate.
\end{proof}

From estimate \eqref{eqMaekawaIneq} and inequality \eqref{eqEstimateGradv} result the following corollary.

\begin{corollary}\label{corMaekawaIneq}
The solution $v$ of \eqref{eqv} satisfies the following energy estimate
\begin{equation}\label{eqMaekawaIneqCor}
\|\nabla v(t)\|_{L^2}^2 + \frac{1}{2}\int_0^t \|\Delta v(s)\|_{L^2}^2\,ds 
\leqslant C\left(\|\nabla v_0\|_{L^2}^2 + 1\right) (1+t)^C
\end{equation}
for all $t\geqslant0$, where the constant $C = C(L,a,\|v_0\|_{L^2}) > 0$ is an increasing function of $a$ and the initial data $\|v_0\|_{L^2}$.

\end{corollary}

\begin{proof}
Let us consider estimate \eqref{eqEstimateGradv} and focus on the right-hand side of the inequality. Given Proposition \ref{propMaekawaIneq}, for all $t\geqslant0$ we have that $\int_0^t \|\nabla v(s)\|_{L^2}^2\,ds \leqslant C_1+C_1\ln(1+t)$ for some constant $C_1 = C_1(L,a,\|v_0\|_{L^2}) > 0$. Defining $K(s) = 2\frac{C^2|a|^2}{1+s} + 2\frac{C_0}{L}\|\nabla v(s)\|_{L^2}^2$ as in \eqref{eqdefK}, we thus have $\int_0^t K(s)ds \leqslant C_2+C_2\ln(1+t)$ for some constant $C_2 = C_2(L,a,\|v_0\|_{L^2}) > 0$ and therefore
\begin{multline*}
\|\nabla v_0\|_{L^2}^2 \,e^{\int_0^t K(\sigma)d\sigma} + \int_0^t e^{\int_s^t K(\sigma)d\sigma}\frac{2C^2|a|^2}{(1+s)^2}\|v(s)\|_{L^2}^2\,ds\\
\leqslant e^{C_2}\|\nabla v_0\|_{L^2}^2 \,(1+t)^{C_2} + e^{C_2}(1+t)^{C_2}\int_0^t \frac{2C^2|a|^2}{(1+s)^2}\|v(s)\|_{L^2}^2\,ds.
\end{multline*}
Finally, according to \eqref{eqMaekawaIneq} again, for all $s\geqslant0$ one has $\|v(s)\|_{L^2}^2 \leqslant C_1+C_1\ln(1+s)$ so
\[
\forall t\geqslant0\qquad
\int_0^t \frac{2C^2|a|^2}{(1+s)^2}\|v(s)\|_{L^2}^2\,ds
\leqslant \int_0^{+\infty} \frac{2C^2|a|^2}{(1+s)^2}\|v(s)\|_{L^2}^2\,ds
\leqslant C_3(L,a,\|v_0\|_{L^2}),
\]
which completes the proof.
\end{proof}

\begin{corollary}\label{corCorMaekawaIneq}
Let $u$ be the solution to \eqref{eqNSdimensionless} with initial data $u_0\in X$. The exists a constant $C>0$ depending on $L$, $a$ and $u_0$ such that for all $t\geqslant0$
\begin{equation}\label{eqMaekawaIneqCorCorGradu}
\int_0^t \|\nabla u(s)\|_{L^2}^2\,ds 
\leqslant C(1 + \ln(1+t)).
\end{equation}
and
\begin{equation}\label{eqMaekawaIneqCorCorDeltau}
\|\nabla u(t)\|_{L^2}^2 + \frac{1}{2}\int_0^t \|\Delta u(s)\|_{L^2}^2\,ds 
\leqslant C (1+t)^C.
\end{equation}
\end{corollary}

\begin{proof}
Given Proposition \ref{propMaekawaIneq}, given that $u = v + a\ulo$ and that $\|\nabla\ulo(s)\|_{L^2}^2 = C'\frac{L}{1+s}$ where $C'>0$ is a universal constant, one has for all $t\geqslant0$ 
\[
\int_0^t \|\nabla u(s)\|_{L^2}^2\,ds \leqslant C_1(1+\ln(1+t))
\]
where $C_1>0$ depends on $L$, $a$ and $u_0$.
Moreover, given that $\|\Delta\ulo(s)\|_{L^2}^2 = C''\frac{L}{(1+s)^2}$ for some universal constant $C''>0$, the previous corollary implies for all $t\geqslant0$ 
\[
\|\nabla u(t)\|_{L^2}^2 + \frac{1}{2}\int_0^t \|\Delta u(s)\|_{L^2}^2\,ds 
\leqslant C_2 (1+t)^{C_2}
\]
where $C_2>0$ depends on $L$, $a$ and $u_0$.
\end{proof}

Following the approach from \cite{GallayRoussierMichon}, 
we will now decompose $u$ and $v$ as the sum of a vertical mean value plus a zero-vertical-mean component.
Let us recall the notations introduced in \eqref{eqdefQ} and \eqref{eqdefumuv}, and that $1-Q$ is hence the orthogonal projection operator on the vector fields of $(H^1(\Omega))^3$ who have a zero mean value relative to the vertical coordinate.
One can see that $Q$ is also well defined on $\ulo(t)$ for all $t\geqslant0$ and that $Q\ulo(t) = \ulo(t)$. For every $u\in X$, we thus have $u = \um + \uv$ with $\um = \vm + a\ulo \in X$ and $\uv = \vv\in H$. 

Because $\uv = (1-Q)u$ has zero vertical mean value, the inequality of Poincaré implies that its $L^2$-norm and the $L^2$-norm of its derivatives decay exponentially towards zero. This is the object of the following proposition.

\begin{lemma}[Poincaré's inequality]\label{ThPoincare}
One has $\|\vv\|_{L^2} \leqslant L \|\nabla \vv\|_{L^2}$ for any vector field $v$ in $\H$.
\end{lemma}

\begin{proposition}[Energy estimates for the zero-vertical-mean velocity field]\label{lemmaEstimatesuv}
Let $u$ be the solution to \eqref{eqNSdimensionless} with initial condition $u_0\in X$. We have the following energy estimates. 
\begin{enumerate}
\item[(i)]
For all $t>0$, 
\begin{equation}\label{eqEstimateuv}
\|\uv(t)\|_{L^2}^2 + \frac{1}{2}\int_0^t e^{-\frac{(t-s)}{L^2} + \int_s^t k\ver}\ \|\nabla\uv(s)\|_{L^2}^2\,ds \leqslant \|\uv_0\|_{L^2}^2\ e^{-\frac{t}{L^2} + \int_0^t k\ver}
\end{equation}
where $k\ver(t) = \frac{2C_0}{L}\|\nabla \um(t)\|_{L^2}^2$, $C_0$ being the universal constant from the Ladyzhenskaya inequality of Lemma \ref{lemmaLadyzhen}.
\item[(ii)]
For all $t>0$, 
\begin{equation}\label{eqEstimateGraduv}
\|\nabla\uv(t)\|_{L^2}^2 + \frac{1}{2}\int_0^t e^{-\frac{(t-s)}{L^2} + \int_s^t K\ver}\ \|\Delta\uv(s)\|_{L^2}^2\,ds \leqslant C\, e^{-\frac{t}{L^2} + C\ln(1+t)}
\end{equation}
where $K\ver(t) = \frac{8\,C_0}{L}\|\nabla u(t)\|_{L^2}^2$ and $C>0$ depends on $L$, $a$ and $u_0$.
\item[(iii)]
For all $t\geqslant 1$,
\begin{equation}\label{eqEstimateLaplaceuv}
\|\Delta\uv(t)\|_{L^2}^2 + \frac{1}{2}\int_1^t e^{-\frac{(t-s)}{L^2} + \int_s^t \mathcal{K}\ver}\ \|\nabla\Delta\uv(s)\|_{L^2}^2\,ds \leqslant C'\, e^{-\frac{t}{L^2} + C'\ln(1+t)}
\end{equation}
where $\mathcal{K}\ver(t) = \frac{36\,C_0}{L}\|\nabla u(t)\|_{L^2}^2$ and $C'>0$ depends on $L$, $a$ and $u_0$.
\end{enumerate}
\end{proposition}

\refstepcounter{theorem}
\paragraph{Remark \thetheorem.}\label{rkMajorationkKK}
Given that $Q$ is an orthogonal projection, we get that $\|\nabla u\|_{L^2}^2 = \|\nabla\um\|_{L^2}^2 + \|\nabla\uv\|_{L^2}^2$ and $\|\Delta u\|_{L^2}^2 = \|\Delta\um\|_{L^2}^2 + \|\Delta\uv\|_{L^2}^2$ so estimates \eqref{eqMaekawaIneqCorCorGradu} and \eqref{eqMaekawaIneqCorCorDeltau} from corollary \ref{corCorMaekawaIneq} hold as well for $\um$ and $\uv$.
In particular, for all $t\geqslant0$ we have
\begin{equation}\label{eqMajorationK}
\int_0^t k\ver(s)\,ds + \int_0^t K\ver(s)\,ds + \int_0^t \mathcal{K}\ver(s)\,ds \leqslant C(1+\ln(1+t))
\end{equation} 
and 
\begin{equation}\label{eqMaekawaIneqCorCorDeltaum}
\|\nabla\um(t)\|_{L^2}^2 + \frac{1}{2}\int_0^t \|\Delta \um(s)\|_{L^2}^2\,ds 
\leqslant C (1+t)^C
\end{equation} 
where $C>0$ depends on $L$, $a$ and $u_0$.

\begin{proof}
Given that $(1-Q)\um = 0$ and $(1-Q)\uv = \uv$, by applying $1-Q$ to \eqref{eqNSdimensionless} we get that $\uv$ satisfies the equation
\begin{equation}\label{equv}
\partial_t\uv + (1-Q)(\uv\cdot\nabla\uv) + \um\cdot\nabla\uv + \uv\cdot\nabla\um = \Delta\uv - \nabla \pv
\end{equation}
where we note $\pv = (1-Q)p = p(t) - \frac{1}{2\pi L}\int_0^{2\pi L} p(t,\cdot,\cdot,z)\,dz$.
The calculations that lead to estimates \eqref{eqEstimateuv} to \eqref{eqEstimateLaplaceuv} are similar to the ones leading to \eqref{eqEstimatev}. Let us consider in $(L^2(\Omega))^3$ the scalar product of \eqref{equv} against $\uv$, where
\[
\int_\Omega \uv(t)\cdot\partial_t\uv(t) = \int_\Omega \partial_t\left(\frac{1}{2}\uv(t)^2\right) = \frac{1}{2}\partial_t\|\uv(t)\|_{L^2}^2
\quad\text{ and }\quad
\int_\Omega\uv(t)\cdot\Delta\uv(t) = -\|\nabla\uv(t)\|_{L^2}^2
\]
by doing an integration by parts to get the right-hand-side equality, where
\[
\int_\Omega\uv(t)\cdot\nabla\pv(t) = 0 
\quad\text{ and }\quad
\int_\Omega\uv(t)\cdot(\um(t)\cdot\nabla)\uv(t) = 0
\]
by integrating by parts and because $\di(\uv(t)) = \di(\um(t)) = 0$, and where
\[
\int_\Omega\uv(t)\cdot(1-Q)(\uv(t)\cdot\nabla\uv(t))
= \int_\Omega\uv(t)\cdot(\uv(t)\cdot\nabla\uv(t))
= 0
\]
given that $\int_\Omega\uv(t)\cdot Q(\uv(t)\cdot\nabla\uv(t)) = 0$ and because  $\di(\uv(t)) = 0$. 
To treat the last term, we use first the Cauchy-Schwarz inequality then Ladyzhenskaya's inequality (lemma \ref{lemmaLadyzhen}) then Young's inequality to write
\begin{multline*}
\left|\int_\Omega\uv(t)\cdot(\uv(t)\cdot\nabla)\um(t)\right| 
\leqslant \|\uv(t)\|_{L^4}^2 \|\nabla\um(t)\|_{L^2}
\\
\leqslant \left(\!\frac{C_0}{L}\!\right)^{\frac{1}{2}} \|\nabla\uv(t)\|_{L^2}\|\uv(t)\|_{L^2} \|\nabla\um(t)\|_{L^2}
\leqslant \frac{1}{4}\|\nabla\uv(t)\|_{L^2}^2 + \frac{C_0}{L}\,\|\nabla\um(t)\|_{L^2}^2 \|\uv(t)\|_{L^2}^2,
\end{multline*}
which gives for every $s>0$
\begin{equation}
\frac{1}{2}\partial_s\|\uv(s)\|_{L^2}^2 \leqslant -\frac{3}{4}\|\nabla\uv(s)\|_{L^2}^2 + \frac{C_0}{L}\,\|\nabla\um(s)\|_{L^2}^2 \|\uv(s)\|_{L^2}^2.
\end{equation}
Using Poincaré's inequality to write 
\begin{equation}
\partial_s\|\uv(s)\|_{L^2}^2 \leqslant -\frac{1}{2}\|\nabla\uv(s)\|_{L^2}^2  - \frac{1}{L^2}\|\uv(s)\|_{L^2}^2 + k\ver(s)\, \|\uv(s)\|_{L^2}^2,
\end{equation}
multiplying then this whole inequality by $\exp(\int_0^s\frac{1}{L^2} - k\ver(\sigma)\,d\sigma)$ and integrating it for $s$ between 0 and $t>0$ leads to estimate \eqref{eqEstimateuv}.

We will now show estimate \eqref{eqEstimateGraduv} with the same approach.
Let us consider in $(L^2(\Omega))^3$ the scalar product of \eqref{equv} against $\Delta\uv$. 
Let us see that $\int_\Omega \Delta\uv(t)\cdot\partial_t\uv(t) = -\frac{1}{2}\partial_t\|\nabla\uv(t)\|_{L^2}^2$ and that $\int_\Omega\Delta\uv(t)\cdot\nabla\pv(t) = 0$ by integrating by parts and for the latter equality because $\mbox{$\di(\Delta\uv(t)) = 0$}$.
Given that $\di\um = 0$, we get thanks to the Cauchy-Schwaz inequality then Ladyzhenskaya's inequality then Young's inequality that
\begin{multline*}
\left|\int_\Omega\Delta\uv(t)\cdot(\um(t)\cdot\nabla)\uv(t)\right| 
\leqslant \|\nabla\uv(t)\|_{L^4}^2 \|\nabla\um(t)\|_{L^2}
\\
\leqslant \left(\!\frac{C_0}{L}\!\right)^{\frac{1}{2}} \|\Delta\uv(t)\|_{L^2}\|\nabla\uv(t)\|_{L^2} \|\nabla\um(t)\|_{L^2}
\leqslant \frac{1}{16}\|\Delta\uv(t)\|_{L^2}^2 + \frac{4C_0}{L}\,\|\nabla\um(t)\|_{L^2}^2 \|\nabla\uv(t)\|_{L^2}^2.
\end{multline*}
Because $\int_\Omega\Delta\uv(t)\cdot Q(\uv(t)\cdot\nabla\uv(t)) = 0$ and given that $\di(\uv(t)) = 0$, we have in the same way 
\begin{multline*}
\left|\int_\Omega\Delta\uv(t)\cdot(1-Q)(\uv(t)\cdot\nabla\uv(t))\right| 
\leqslant \|\nabla\uv(t)\|_{L^4}^2 \|\nabla\uv(t)\|_{L^2}
\\
\leqslant \left(\!\frac{C_0}{L}\!\right)^{\frac{1}{2}} \|\Delta\uv(t)\|_{L^2}\|\nabla\uv(t)\|_{L^2}^2
\leqslant \frac{1}{16}\|\Delta\uv(t)\|_{L^2}^2 + \frac{4C_0}{L}\,\|\nabla\uv(t)\|_{L^2}^4.
\end{multline*}
The same reasoning gives finally
\begin{multline*}
\left|\int_\Omega\Delta\uv(t)\cdot(\uv(t)\cdot\nabla)\um(t)\right| 
\leqslant \|\Delta\uv(t)\|_{L^2} \|\uv(t)\|_{L^4} \|\nabla\um(t)\|_{L^4}
\\
\leqslant \frac{1}{8}\|\Delta\uv(t)\|_{L^2}^2 + \frac{2C_0}{L}\,\|\nabla\uv(t)\|_{L^2}\|\uv(t)\|_{L^2} \|\Delta\um(t)\|_{L^2}\|\nabla\um(t)\|_{L^2}.
\end{multline*}

All of the above calculation give the following inequality for every $s>0$
\begin{equation}
\frac{1}{2}\partial_s\|\nabla\uv(s)\|_{L^2}^2 + \frac{3}{4}\|\Delta\uv(s)\|_{L^2}^2 
\leqslant \frac{1}{2}K\ver(s)\|\nabla\uv(s)\|_{L^2}^2 + \frac{2C_0}{L}K\ver_1(s)
\end{equation}
where $K\ver_1(s) = \|\nabla\uv(s)\|_{L^2}\|\uv(s)\|_{L^2} \|\Delta\um(s)\|_{L^2}\|\nabla\um(s)\|_{L^2}$.
Using Poincaré's inequality on $\nabla\uv$ to write 
\begin{equation}
\partial_s\|\nabla\uv(s)\|_{L^2}^2 + \frac{1}{2}\|\Delta\uv(s)\|_{L^2}^2 + \frac{1}{L^2}\|\nabla\uv(s)\|_{L^2}^2
\leqslant K\ver(s)\|\nabla\uv(s)\|_{L^2}^2 + \frac{4C_0}{L}K\ver_1(s),
\end{equation}
multiplying then this whole inequality by $\exp(\int_0^s\frac{1}{L^2} - K\ver(\sigma)\,d\sigma)$ and integrating it for $s$ between 0 and $t>0$ leads to the relation
\begin{multline}\label{local2}
\|\nabla\uv(t)\|_{L^2}^2 + \frac{1}{2} \int_0^t e^{-\frac{t-s}{L^2} + \int_s^tK\ver} \|\Delta\uv(s)\|_{L^2}^2\,ds \\
\leqslant \|\nabla\uv_0\|_{L^2}^2 e^{-\frac{t}{L^2} + \int_0^tK\ver} + \int_0^t e^{-\frac{t-s}{L^2} + \int_s^tK\ver}\, \frac{4C_0}{L}K\ver_1(s)\,ds.
\end{multline}
Let us focus on the right-hand side of \eqref{local2}. Given \eqref{eqMajorationK} we have that
\(
e^{\int_0^tK\ver} 
\leqslant C_1(1+t)^{C_1}
\)
for some constant $C_1 = C_1(L,a,u_0)>0$.
Now let us prove that
\(
\int_0^t e^{-\frac{t-s}{L^2} + \int_s^tK\ver}\, K\ver_1(s)\,ds
\leqslant C_2(1+t)^{C_2}e^{-\frac{t}{L^2}}
\)
for some constant $C_2 = C_2(L,a,u_0)>0$. This is a consequence of estimate \eqref{eqEstimateuv} and Remark \ref{rkMajorationkKK}, and arises from writing $K\ver_1 \leqslant \frac{1}{2} \|\nabla\uv\|_{L^2}^2\|\nabla\um\|_{L^2}^2 + \frac{1}{2} \|\Delta\um\|_{L^2}^2\|\uv\|_{L^2}^2$
then 
\begin{multline*}
\int_0^t e^{-\frac{t-s}{L^2} + \int_s^tK\ver}\, \|\nabla\uv(s)\|_{L^2}^2\|\nabla\um(s)\|_{L^2}^2\,ds
\leqslant \int_0^t e^{-\frac{t-s}{L^2} + \int_0^tK\ver}\, \|\nabla\uv(s)\|_{L^2}^2\, C_3(1+s)^{C_3} \,ds
\\
\leqslant C_3(1+t)^{C_3} e^{\int_0^tK\ver} \int_0^t e^{-\frac{t-s}{L^2}}\, \|\nabla\uv(s)\|_{L^2}^2 \,ds
\leqslant C_4(1+t)^{C_4}\  C_5\,e^{-\frac{t}{L^2} + \int_0^tk\ver}
\leqslant C_6(1+t)^{C_6}e^{-\frac{t}{L^2}}
\end{multline*}
and finally 
\begin{multline*}
\int_0^t e^{-\frac{t-s}{L^2} + \int_s^tK\ver}\, \|\Delta\um(s)\|_{L^2}^2\|\uv(s)\|_{L^2}^2\,ds
\leqslant \int_0^t e^{-\frac{t-s}{L^2} + \int_0^tK\ver}\, \|\Delta\um(s)\|_{L^2}^2\, C_7\,e^{-\frac{s}{L^2} + \int_0^sk\ver} \,ds
\\
\leqslant C_7\,e^{-\frac{t}{L^2} + \int_0^tK\ver + \int_0^tk\ver} \int_0^t \|\Delta\um(s)\|_{L^2}^2 \,ds
\leqslant C_8(1+t)^{C_8}e^{-\frac{t}{L^2}}\  C_9(1+t)^{C_9} 
\leqslant C_{10}(1+t)^{C_{10}}e^{-\frac{t}{L^2}}
\end{multline*}
for some postive constants $C_3$ to $C_{10}$ which depend on $L$, $a$ and $u_0$.
Inequality \eqref{local2} combined with these last calculations gives estimate \eqref{eqEstimateGraduv}.

Estimate \eqref{eqEstimateLaplaceuv} is obtained in the same way, by writing the scalar product of \eqref{equv} against the bilaplacian $\Delta^2\uv$.
The details are left to the reader.
\end{proof}

\section{Asymptotic Behaviour}
\label{SectionAsymptotics}

Knowing from the previous section that the zero-vertical-mean component $\uv$ of $u$ and its derivatives decay exponentially in $(L^2(\Omega))^3$ as $t \to +\infty$, we want now to show that $\|\um(t) - a\ulo(t)\|_{L^2} \to 0$ and $\|\nabla\um(t) - a\nabla\ulo(t)\|_{L^2} = o(t^{-\frac{1}{2}})$ when $t\to+\infty$.

To do so, let us write the equation satisfied by $\vm = \um - a\ulo$. We first apply $Q$ to \eqref{eqNSdimensionless} and get 
\[
\partial_t\um + Q(u\cdot\nabla u) = \Delta\um - \nabla\pmm,
\]
where we can subtract the quantity $a\partial_t\ulo = a\Delta\ulo$ from both sides of the equation and where
\[
Q(u\cdot\nabla u) = Q(\um\cdot\nabla\um) + Q(\um\cdot\nabla\uv) + Q(\uv\cdot\nabla\um) + Q(\uv\cdot\nabla\uv) = \um\cdot\nabla\um + Q(\uv\cdot\nabla\uv),
\]
which leads to
\[
\partial_t\vm + \um\cdot\nabla\um + Q(\uv\cdot\nabla\uv) = \Delta\vm - \nabla\pmm.
\]
Let us see here that $\um\cdot\nabla\um$ is a gradient. Because $u$ is helical, $\um$ is a radial two-and-a-half dimensional flow $\um(t,r,\theta,z) = \um_r(t,r)e_r + \um_\theta(t,r)e_\theta + \um_z(t,r)e_z$; and given that $\di \um = \frac{1}{r}\partial(r\,\um_r(r)) = 0$ we know that $\um_r = 0$. Hence 
\[
\um\cdot\nabla\um = \frac{1}{r}\um_\theta\,\partial_\theta\um = -\frac{1}{r}\,\um_\theta^2\,e_r = - \nabla\left(\int_1^r \frac{1}{\rho}\um_\theta(t,\rho)^2\,d\rho\right)
\]
and so
\[
\partial_t\vm + Q(\uv\cdot\nabla\uv) = \Delta\vm - \nabla\left( \pmm - \int_1^r \frac{1}{\rho}\um_\theta(t,\rho)^2\,d\rho\right).
\]

Taking the Leray projection $P$ of that last equation, we see that $\vm$ satisfies the heat equation with source term
\begin{equation}\label{eqvm}
\partial_t\vm + \nm = \Delta\vm
\end{equation}
where 
\begin{equation}\label{defN}
\nm = PQ(\uv\cdot\nabla\uv) = PQ(\vv\cdot\nabla\vv)
\end{equation}
decays exponentially in time according to Proposition \ref{lemmaEstimatesuv}.
Indeed, because $P$ and $Q$ are both orthogonal projections one can write $\|\nm\|_{L^2} = \|\uv\cdot\nabla\uv\|_{L^2}$, then $\|\uv\cdot\nabla\uv\|_{L^2} \leqslant \|\uv\|_{L^4} \|\nabla\uv\|_{L^4}$ by Hölder's inequality and $\|\uv\|_{L^4} \|\nabla\uv\|_{L^4} \leqslant (\frac{C_0}{L})^{\frac{1}{2}} \|\uv\|_{L^2}^{\frac{1}{2}} \|\nabla\uv\|_{L^2} \|\Delta\uv\|_{L^2}^{\frac{1}{2}}$ by Ladyzhenskaya's inequality (Lemma \ref{lemmaLadyzhen}). Proposition \ref{lemmaEstimatesuv} therefore ensures that there exist two positive constants $C>0$ and $C'>0$ (depending on $L$ , $a$ and $v_0$) such that 
\begin{equation}\label{eqNdecroitExponentiellementL2}
\|\nm(t)\|_{L^2} \leqslant C'e^{-Ct}
\qquad \forall t\geqslant1.
\end{equation}

Let us now remind briefly some properties of the heat semigroup in $L^2$, then we shall complete the proof of Theorem \ref{ThIntro}.

\begin{lemma}[Heat semigroup properties]\label{lemmaSemigp}
Let $(S(t))_{t\geqslant0}$ be the heat semigroup on $\Omega$. Let $\alpha\in\NN^3$. There exists a constant $C>0$ such that 
\[
\|\partial_\alpha S(t)v\|_{L^2} \leqslant C\,t^{-\frac{|\alpha|}{2}}\|v\|_{L^2}
\]
for any $t>0$ and any vector field $v\in(L^2(\Omega))^3$.
Moreover, for any vector field $v\in(L^2(\Omega))^3$ one has $t^{\frac{|\alpha|}{2}}\,\|\partial_\alpha S(t)v\|_{L^2} \longrightarrow0$ when $t\to+\infty$.
\end{lemma}

\begin{proposition}[Time-asymptotic behaviour]\label{ThAsymptoticOseen}
For every $v_0 \in H$, the component $\vm = Qv$ of the solution $v$ of \eqref{eqv} with initial data $v_0$ given by Theorem \ref{ThGlobal} satisfies $\|\vm(t)\|_{L^2} + \sqrt{t}\,\|\nabla\vm(t)\|_{L^2} \to 0$ when $t$ tends to infinity.
\end{proposition}

\begin{proof}
Given \eqref{eqvm}, $\vm$ satisfies the integral equation
\begin{equation}\label{eqvmIntegral}
\vm(t) = S(t)\vm_0 - \int_0^tS(t-s)\nm(s)\,ds
\qquad\forall t\geqslant 0
\end{equation}
where $(S(t))_{t\geqslant0}$ denotes the heat semigroup in $\Omega$.
Because $\partial_z\vm = 0$, to prove the theorem we need to show that $t^{\frac{|\alpha|}{2}}\|\partial_\alpha\vm(t)\|_{L^2} \to 0$ for $\alpha$ being the multi-index $(0,0,0)$, $(1,0,0)$ or $(0,1,0)$ in Cartesian coordinates.

Fist of all, let us consider $t_0\geqslant1$ and $t\geqslant 2t_0$ and let us rewrite \eqref{eqvmIntegral} as
\[
\vm(t) = S(t-t_0)\vm(t_0) - \int_{t_0}^{t/2}S(t-s)\nm(s)\,ds - \int_{t/2}^tS(t-s)\nm(s)\,ds.
\]
Let us then treat each term separately.

The fact that $t^{\frac{|\alpha|}{2}}\|\partial_\alpha S(t-t_0)\vm(t_0)\|_{L^2} \to 0$ when $t\to+\infty$ is direct from Lemma \ref{lemmaSemigp}. 
On the other hand, given Lemma \ref{lemmaSemigp} and \eqref{eqNdecroitExponentiellementL2} we can write that
\begin{multline*}
\int_{t_0}^{t/2} t^{\frac{|\alpha|}{2}}\|\partial_\alpha S(t-s)\nm(s)\|_{L^2} \,ds 
\leqslant \int_{t_0}^{t/2} t^{\frac{|\alpha|}{2}} (t-s)^{-\frac{|\alpha|}{2}} \|\nm(s)\|_{L^2} \,ds 
\\ 
\leqslant 2^{\frac{|\alpha|}{2}}\int_{t_0}^{t/2}  C'e^{-Cs} \,ds 
\leqslant 2^{\frac{|\alpha|}{2}} \frac{C'}{C} e^{-Ct_0}
\end{multline*}
and that 
\begin{multline*}
\int_{t/2}^t t^{\frac{|\alpha|}{2}}\|\partial_\alpha S(t-s)\nm(s)\|_{L^2} \,ds 
\leqslant \int_{t/2}^t t^{\frac{|\alpha|}{2}} (t-s)^{-\frac{|\alpha|}{2}} C'e^{-Cs} \,ds 
\\
\leqslant C't \left(\int_{1/2}^1 s^{\frac{|\alpha|}{2}}(1-s)^{-\frac{|\alpha|}{2}} \,ds\right) e^{-Ct/2}
\leqslant \sqrt{2}\,C't\,e^{-Ct/2} \xrightarrow[t\to+\infty]{} 0,
\end{multline*}
where $C$ and $C'$ are the constants from \eqref{eqNdecroitExponentiellementL2}.
We have therefore, for all $t_0\geqslant 1$
\[
\underset{t\to+\infty}\limsup\ t^{\frac{|\alpha|}{2}}\|\partial_\alpha\vm(t)\|_{L^2} \leqslant 2^{\frac{|\alpha|}{2}} \frac{C'}{C} e^{-Ct_0}
\xrightarrow[t_0\to+\infty]{} 0,
\]
which ends the proof.
\end{proof}

\section{Additional Results and Comments}\label{SectionConclusion}

The Navier-Stokes equations \eqref{eqNSdimensionless} can also be studied from the point of view of the vorticity. This alternative approach has a number of advantages, and several results have a clearer formulation or a simpler proof when formulated in terms of the vorticity, see for example \cite{GalWayR2,GalWayR3,MajBert,JiuLoNiuLo}. Regarding Theorem~\ref{ThIntro}, it is interesting to note that the assumption $u_0 \in X$ is automatically satisfied when the vorticity $\omega_0 = \rot(u_0)$ is locally square integrable and decays sufficiently rapidly at infinity. To make a precise statement, following \cite{GalWayR2} we introduce for any real number $m \ge 0$ the weighted Lebesgue space $L^2_m(\Omega) = \{w \in L^2(\Omega)\mid \|w\|_{L^2_m} < \infty\}$ where 
\[
\|w\|_{L^2_m}^2 = \int_\Omega (1+r^2)^m|w(r,\theta,z)|^2\,rdrd\theta dz.
\]
Using H\"older's inequality, it is straightforward that $L^2_m(\Omega)\hookrightarrow L^1(\Omega)$ when $m > 1$. The corresponding function space for helical vector fields is
\[
Y_m = \left\{w\in \left(L_m^2(\Omega)\right)^3\ \middle|\ \di(w)=0,\ w\text{ is helical} \right\}.
\]
If $\omega \in Y_m$, we denote by $u = \BS[\omega]$ the unique helical vector field $u \in (L^6(\Omega))^3$ satisfying $\di(u) = 0$ and $\rot(u) = \omega$. The correspondence $\omega \mapsto \BS[\omega]$ is the helical Biot-Savart law in the domain $\Omega$, see for example \cite{MajBert}. 

\begin{proposition}\label{YmProp}
Let $\omega \in Y_m$ for some $m > 1$, and let us define $a \in \RR$ by 
\begin{equation}\label{adef}
a = \frac{1}{2\pi L}\int_{\Omega}\omega_z(r,\theta,z)\,rdrd\theta dz.
\end{equation}
Then the velocity field $u$ associated with $\omega$ can be decomposed as $u = a\ulo(0) + v$ where $v \in H$ satifies $\|v\|_{H^1} \leqslant C\|\omega\|_{L^2_m}$ for some constant $C>0$. In particular, $u$ belongs to $X$. 
\end{proposition}

\begin{proof}
As in \eqref{eqdefumuv}, we denote $\um = Qu$, $\uv = (1-Q)u$, and similarly $\wm = Q\omega$, $\wv = (1-Q)\omega$. For the components with zero vertical mean, the Biot-Savart law $\uv = \BS[\wv]$ has a convenient expression in Fourier variables from which one easily deduces that $\|\uv\|_{H^1} \leqslant C_1 \|\wv\|_{L^2}$ for some constant $C_1>0$, see \cite[proposition A.2]{GallayRoussierMichon}. 
On the other hand, by helical symmetry and incompressibility, the vertically averaged vector fields $\um$ and $\wm$ take the form
\begin{equation}\label{umom}
\um = \um_\theta(r)e_\theta + \um_z(r)e_z
\qquad\text{ and }\qquad
\wm = \wm_\theta(r)e_\theta + \wm_z(r)e_z\,,
\end{equation}
and the relation $\wm = \rot(\um)$ reduces to $\wm_\theta = -\partial_r \um_z$ and $\wm_z = \frac{1}{r}\partial_r(r\um_\theta)$. This gives the explicit formulas
\begin{equation}\label{BSbar}
\um_\theta(r) = \frac{1}{r}\int_0^r \wm_z(\rho) \,\rho d\rho
\qquad\text{ and }\qquad
\um_z(r) = \int_r^{+\infty} \wm_\theta(\rho) \,d\rho
\end{equation}
for every $r>0$. We know from Fourier analysis that $\|\nabla \um\|_{L^2} \leqslant C_2\|\wm\|_{L^2}$ for some constant $C_2>0$, but the difficulty is to estimate the $L^2$ norm of $\um$. 

We first consider the vertical velocity $\um_z$. Using H\"older's inequality, one can bound
\begin{equation}\label{YmPropumz}
|\um_z(r)| 
\leqslant \int_r^\infty |\wm_\theta(\rho)| \,d\rho 
\leqslant \|\wm_\theta\|_{L^2_m} \left(\int_r^{+\infty} \frac{1}{\rho(1+\rho^2)^m}\,d\rho\right)^{\frac{1}{2}}
\leqslant C_3\, \|\wm_\theta\|_{L^2_m} \frac{1 + \left(\ln_+(\frac{1}{r})\right)^{\frac{1}{2}}}{(1+r)^{m}}
\end{equation}
where $\ln_+(s) = \max\bigl(\ln(s),0\bigr)$ and $C_3>0$ is a constant depending on $m$. In particular, since $m > 1$, we see that $\|\um_z\|_{L^2} \leqslant C_3' \|\wm_\theta\|_{L^2_m}$ for some $C_3'>0$ depending on $m$.

Let us next estimate the azimuthal velocity $\um_\theta$, which requires a more careful argument. Indeed, even if $\wm_z$ decays rapidly at infinity, the first expression in \eqref{BSbar} shows that $\um_\theta \sim \frac{a}{2\pi r}$ as $r \to +\infty$, so that $\um_\theta \notin L^2(\Omega)$ if the constant $a$ defined in \eqref{adef} is nonzero. This is precisely the reason for which one has to extract from $\um_\theta$ a multiple of the Oseen vortex. We thus define
\begin{equation}\label{wmdecomp}
\vm_\theta(r) = \um_\theta(r) - \frac{a}{2\pi r}\left(1 - e^{-r^2/4}\right)
\qquad\text{ and }\qquad
\wwm_z(r) = \wm_z(r) - \frac{a}{4\pi}e^{-r^2/4}\,,
\end{equation}
so that $\vm_\theta e_\theta = \BS[\wwm_z e_z]$ is the velocity field associated with $\wwm_z e_z$ via the Bio-Savart law. 
The definition of $a$ in \eqref{adef} implies that $\int_0^\infty \wwm_z(r)\,rdr = 0$, so that $\forall r>0$
\begin{equation}\label{BSbar2}
\vm_\theta(r) = \frac{1}{r}\int_0^r \wwm_z(\rho) \,\rho d\rho = -\frac{1}{r}\int_r^{+\infty} \wwm_z(\rho) \,\rho d\rho.
\end{equation}
Using H\"older's inequality again we deduce that
\begin{equation}\label{YmPropvmtheta}
r |\vm_\theta(r)| 
\leqslant \|\wwm_z\|_{L^2_m} \min\left( \int_0^r \frac{\rho}{(1+\rho^2)^m}\,d\rho \,,\, \int_r^{+\infty} \frac{\rho}{(1+\rho^2)^m}\,d\rho\right)^{1/2} 
\leqslant \frac{C_4\,r\,\|\wwm_z\|_{L^2_m}}{(1+r)^m},
\end{equation}
hence $\|\vm_\theta\|_{L^2(\Omega)} \leqslant C_4' \|\wwm_z\|_{L^2_m} \leqslant C_4''\|\wm_z\|_{L^2_m}$ where $C_4,C_4',C_4''>0$ are some constants depending on $m$.

Summarizing, in view of \eqref{umom} and \eqref{wmdecomp} we have shown that $u = \BS[\omega] = a\ulo(0) + v$ where
\[
v(r,\theta,z) = \vm_\theta(r)e_\theta + \um_z(r)e_z + \uv(r,\theta,z)
\]
satisfies $\|v\|_{H^1} \leqslant C\|\omega\|_{L^2_m}$.
This concludes the proof. 
\end{proof}

It follows from Proposition~\ref{YmProp} that Theorem~\ref{ThIntro} applies when the initial velocity $u_0$ is helical, incompressible and satisfies $\omega_0 = \rot(u_0) \in Y_m$ for some $m > 1$.
In the spirit of the previous works \cite{GalWayR2,GalWayR3,GallayRoussierMichon}, it is legitimate to ask if the convergence result \eqref{eqThIntro} can be improved in that situation.
The answer is positive, but the question is of moderate interest because, as is shown in Section~\ref{SectionAsymptotics}, the long-time asymptotics of the helical Navier-Stokes equations are described by radially symmetric solutions of the heat equation \eqref{eqvm} on $\RR^2$, which are rather easy to analyze. 

A typical result in this direction is the following.

\begin{corollary}\label{YmCor}
Assume that $\omega_0 \in Y_m$ for some $m \in \,]1,2[$ and let $u$ be the solution of the Navier-Stokes equations \eqref{eqNSdimensionless} with initial velocity $u_0 = \BS[\omega_0]$. Then
\begin{equation}\label{uconv2}
\|u(t) - a\ulo(t)\|_{L^2} + \sqrt{t}\,\|\nabla u(t) - a\nabla\ulo(t)\|_{L^2} \,=\, \mathcal{O}\bigl(t^{\frac{1-m}{2}}\bigr) 
\quad \text{ as }\, t \to +\infty\
\end{equation}
where $a$ is defined by \eqref{adef}.
\end{corollary}

We only give a sketch of the proof, and leave the details to the reader. The first step is to verify, as in \cite{GalWayR2}, that the vorticity $\omega = \rot(u)$ stays in the space $Y_m$ for all times. 
The strategy is then to estimate the vertically averaged velocity field $\vm$ as in Proposition~\ref{ThAsymptoticOseen}. Since $\omega_0 \in Y_m$, the bounds \eqref{YmPropumz} and \eqref{YmPropvmtheta} established when proving Proposition~\ref{YmProp} show that the radially symmetric velocity field $\vm_0$ is almost bounded near the origin, and decays like $r^{-m}$ as $r \to \infty$; in particular $\vm_0$ belongs to the weak Lebesgue space $L^q_w(\RR^2)$ with exponent $q = 2/m\in\ ]1,2[\,$. Let $p = 2/(3-m)$, so that $1/p + 1/q = 3/2$. Since $S(t)\vm_0 =
G_t*\vm_0$ where $G_t(x) = (4\pi t)^{-1}\exp(-|x|^2/(4t))$ is the heat kernel
in $\RR^2$, we can apply the generalized Young inequality to obtain the bound
\[
\|S(t)\vm_0\|_{L^2} = \|G_t * \vm_0\|_{L^2} 
\leqslant C \|G_t\|_{L^p} \|\vm_0\|_{L^q_w}
\leqslant C t^{\frac{1-m}{2}} \|\wm_0\|_{L^2_m}
\]
for some constant $C>0$ depending on $m$ only. 
This shows that the first term in the right-hand side of \eqref{eqvmIntegral} decays like $t^{\frac{1-m}{2}}$ in energy norm as $t \to +\infty$, and a similar argument gives the same conclusion for the integral term since estimate \eqref{eqNdecroitExponentiellementL2} tells us that the source term $\nm(t)$ decays exponentially. Thus $\|v(t)\|_{L^2} = \|u(t) - a\ulo(t)\|_{L^2} = \mathcal{O}\bigl(t^{\frac{1-m}{2}}\bigr)$, and proceeding as in Proposition~\ref{ThAsymptoticOseen} we obtain the corresponding result for the first order derivatives too. 

The limitation $m < 2$ in Corollary~\ref{YmCor} is essential, in the sense that the conclusion may fail as soon as $m > 2$. This can be seen by considering helical shear flows of the form 
\begin{equation}\label{explicit1}
u(t,r,\theta,z) = f(t,r) e_z\,, 
\qquad \omega(t,r,\theta,z) = g(t,r) e_\theta\,, 
\end{equation}
where $g = -\partial_r f$ and $f$ solves the linear heat equation $\partial_t f = \partial_r^2 f + \frac{1}{r}\partial_r f$. It is easy to verify that $u$ is an exact solution of the Navier-Stokes equation \eqref{eqNSdimensionless}, where the nonlinearity $(u\cdot\nabla)u + \nabla p$ vanishes identically.
Moreover we have $a = 0$ in \eqref{adef} because $\omega_z = 0$. Therefore, if we take for instance
\[
f(t,r) = \frac{e^{-\frac{r^2}{4(1+t)}}}{4\pi(1+t)}
\qquad\text{ and }\qquad
g(t,r) = \frac{r\,e^{-\frac{r^2}{4(1+t)}}}{8\pi (1+t)^2}\,,
\]
we see that $\omega \in Y_m$ for any $m > 1$ but $t^{1/2}\|u(t)\|_{L^2}$ and $t\|\omega(t)\|_{L^2}$ have nonzero limits as $t \to +\infty$. This proves that \eqref{uconv2} does not hold in general when $m > 2$, even in the finite-energy situation where $a = 0$. 

Another explicit example of helical solutions of \eqref{eqNSdimensionless} is given by two-dimensional vortices of the form 
\begin{equation}\label{explicit2}
u(t,r,\theta,z) = g(t,r) e_\theta\,, 
\qquad \omega(t,r,\theta,z) = h(t,r) e_z\,, 
\end{equation}
where $h = \frac{1}{r}\partial_r(rg)$ solves the linear heat equation $\partial_t h = \partial_r^2 h + \frac{1}{r}\partial_r h$, and $a = 2\pi \int_0^\infty h(0,r)\,rdr$. These solutions can also be used to illustrate the limitations of Corollary~\ref{YmCor}. In a broader perspective, since the source term $\nm$ in the projected equation \eqref{eqvm} decays exponentially as $t \to +\infty$, it can be shown that the long-time asymptotics of the helical Navier-Stokes equations are described by explicit solutions of the form \eqref{explicit1} and \eqref{explicit2}. Note that the situation here is substantially simpler than in the general case considered for instance in \cite{GalWayR2,GalWayR3}, because helical symmetry forces the solutions to become not only independent of $z$ in the long-time limit, but also radially symmetric, which greatly reduces the complexity of the asymptotic dynamics.

\section*{Acknowledgements}

Huge and deep thanks to Thierry Gallay for all the help and the support that led to this article.

\bibliographystyle{plain}
\bibliography{C:/Users/quent/Documents/Softwares/Latex/_mon_test_de_LaTex/BibQV}
%\bibliography{C:/Users/vilaq/Desktop/BibQV}
%\bibliography{/home/vilaq/Bureau/BibQV}

\end{document}